\documentclass[3p,10pt,a4paper,twoside,fleqn,sort&compress]{filomat}
\usepackage{amssymb,amsmath,latexsym}
\usepackage[varg]{pxfonts}
\usepackage{enumitem}

\newtheorem{theorem}{Theorem}[section]

\newtheorem{corollary}[theorem]{Corollary}

\newtheorem{definition}[theorem]{Definition}
\newtheorem{example}[theorem]{Example}

\newtheorem{lemma}[theorem]{Lemma}

\newtheorem{proposition}[theorem]{Proposition}
\newtheorem{remark}[theorem]{Remark}

\newcommand{\FilVol}{xx}
\newcommand{\FilYear}{yyyy}
\newcommand{\FilPages}{zzz--zzz}
\newcommand{\FilDOI}{(will be added later)}
\newcommand{\FilReceived}{Received: dd Month yyyy; Accepted: dd Month yyyy}

\begin{document}

%

\title{Additive and multiplicative properties of Drazin inverse under a new weakly commutativity condition}

\author[affil1]{Rounak Biswas\corref{Rounak Biswas}}
\ead{xyzrounak3@gmail.com}
\author[affil1]{Falguni Roy}
\ead{royfalguni@nitk.edu.in}
\address{National Institute of Technology Karnataka, India}

\newcommand{\AuthorNames}{Rounak Biswas, Falguni Roy}

\newcommand{\FilMSC}{15A09; 47A05; 47L10 }
\newcommand{\FilKeywords}{Drazin inverse; g-Drazin inverse; weakly commutativity; additive properties; multiplicative properties.}
\newcommand{\FilCommunicated}{name of the Editor, mandatory}
\cortext[mycorrespondingauthor]{* Corresponding author}

\begin{abstract}
Given a complex Banach space $X$, let $\mathcal{B}(X)$ be the collection of all bounded linear operators on $X.$ For $A,B\in\mathcal{B}(X)$ we define $A,B$ are $A$-weakly commutative if there exists $C\in\mathcal{B}(X)$ satisfying $$AB=CA\text{ and }BA=AC.$$ The objective of this paper is to study the Drazin invertibility of $A+B$ and  $AB$, when $A, B\in\mathcal{B}(X)^D$ are $A\text{ and }B$-weakly commutative.  Consequently, this extends some of the additive and multiplicative results established by Huanyin and Marjan Sheibani (Linear and Multilinear Algebra \textbf{70.1} (2022): 53-65) for a distinct family of elements.  Moreover, we also establish these additive and multiplicative properties for g-Drazin inverses in a complex Banach algebra.
\end{abstract}

\maketitle

\section{Introduction}
Throughout this paper we consider $X$ as a complex Banach space and $\mathcal{B}(X)$ represents the collection of all bounded linear operators on $X.$
 For $A\in\mathcal{B}(X),$ the notations $\mathcal{N}(A),\text{ }\mathcal{R}(A)$ and $\sigma (A)$ denotes the null space, range space, and spectrum of the operator $A$, respectively. The set of all invertible operators in $\mathcal{B}(X)$ will be denoted by $\mathcal{B}(X)^{inv}$, and $I$ represents the identity operator on $X$.  An operator $A\in \mathcal{B}(X)$ is Drazin invertible if there exists an operator $A^D\in \mathcal{B}(X)$ such that
\begin{equation}\label{def_Drazin}
    AA^D=A^DA,\text{ }A^DAA^D=A^D\text{ and }A(I-AA^D)\in \mathcal{B}(X)^{nil},
    \end{equation} where $\mathcal{B}(X)^{nil}$ represent the collection of all nilpotent operators in $\mathcal{B}(X)$. The operator $A^D$ satisfying (\ref{def_Drazin}) is the Drazin inverse of $A$, and the least non-negative integer $k$ satisfying $\left(A(I-AA^D)\right)^k=0$ is the index of $A,$ denoted by $i(A).$ $\mathcal{B}(X)^D$ stands for the collection of Drazin invertible operators in $\mathcal{B}(X)$ and for $A\in\mathcal{B}(X)^D,$ $A^{\pi}$ denote the spectral idemoptent $I-AA^D.$ If $A\in \mathcal{B}(X)$ is a Drazin invertible operator with index $k$, then the Banach space $X$ has the decomposition \cite[Theorem 12.1.2]{wang2018generalized}
    \begin{equation}\label{decom}
        X=\mathcal{R}(A^k)\oplus\mathcal{N}(A^k)=\mathcal{R}(AA^D)\oplus \mathcal{N}(AA^D),
    \end{equation}
     and with respect to this decomposition of $X,$ $A$ have the block form\cite{deng2008drazin} $$A=A_1\oplus A_2,$$ where $A_1$ is invertible operator on $\mathcal{R}(AA^D)$ and $A_2$ is nilpotent operator on $\mathcal{N}(AA^D)$. A generalization of the Drazin inverse was introduced by Koliha \cite{koliha1996generalized}, known as the generalized Drazin inverse(g-Drazin inverse). From now onwards $\mathcal{A}$ will represent an unital complex Banach algebra. An element $a\in\mathcal{A}$, is said to be g-Drazin invertible if there exists an element $a^d\in\mathcal{A}$ such that $$aa^d=a^da,\text{ }a^daa^d=a^d\text{ and }a-a^2a^d\in\mathcal{A}^{qnil},$$ where $\mathcal{A}^{qnil}$ represent the collection of all quasinilpotent elements in $\mathcal{A}$. Here $a^d$ is the g-Drazin inverse of $a$, and the collection of all g-Drazin invertible elements in $\mathcal{A}$ will be denoted by $\mathcal{A}^d.$
     A fascinating problem in Drazin inverse theory revolves around examining the Drazin invertibility for the sum and product of two Drazin invertible elements. The investigation of the Drazin inverse's additive and multiplicative properties is significantly influenced by the commutative relationship between two elements. In the literature, several results pertaining to this aspect can be found in \cite{deng2008drazin,deng2010new,wei2011note}. Instead of taking commutativity, if we consider the weaker relation $a^2b=ba^2$ for $a,b\in\mathcal{A}^D(\text{or }\mathcal{A}^d)$, then, in general, without any additional condition on $a,b$, it becomes challenging to deduce any conclusions about the Drazin(or g-Drazin) invertibility of $ab$ and $a+b$. Chen and Sheibani \cite{chen2022g} provided several equivalent conditions for the g-Drazin invertibility of $a+b$, assuming that $\lambda a^2b=\lambda^{\prime}ba^2=aba$ and $\mu ab^2=\mu^{\prime}b^2a=bab,$ where $\lambda, \lambda^{\prime},\mu,\mu^{\prime}$ are non zero complex numbers. Recently, Qin and Lu \cite{qin2023drazin} proved these results under weaker conditions $a^2b=aba$ and $b^2a=bab$. In this context, motivated by weakly commutativity relation\cite{pondvelivcek1975weakly}, we introduce a new relation between two elements of a Banach algebra. This relation is defined to generate a distinct family of elements compared to the one used by Chen and Sheibani\cite{chen2022g}.
     \begin{definition}\label{defn_weak} 
    Two elements $a,b\in\mathcal{A}$ are said to be $a$-weakly commutative if there exists $c\in\mathcal{A}$ satisfying $ab=ca\text{ and }ba=ac.$ 
\end{definition}
\noindent Note that commutativity is a particular case of Definition \ref{defn_weak}, when $c=b$. 
\begin{remark}
  According to Definition \ref{defn_weak}, if $a,b$ are $a$-weakly commutative, then $a,c$ are also $a$-weakly commutative.
\end{remark}
\begin{definition}
    For $a,b\in\mathcal{A},$ if $a,b$ are $a$-weakly commutative as well as $b$-weakly commutative also then we say $a,b$ are $\{a,b\}$-weakly commutative.
\end{definition}

\begin{lemma}\label{2commutative}
    If two elements $a,b$ in a Banach algebra $\mathcal{A}$ are $a$-weakly commutative then $a^2b=ba^2.$
    \begin{proof}
        Since $a,b$ are $a$-weakly commutative therefore there exists an element $c\in\mathcal{A}$, such that $ab=ca\text{ and }ba=ac$. Hence $$a^2b=aca=ba^2.$$ 
    \end{proof}
\end{lemma}
\begin{lemma}
    If $a,b\in\mathcal{A}$ are $a$-weakly commutative and $a$ is idempotent, then $a,b$ commutes with each other.
    \begin{proof}
        Since $a$ is idempotent, therefore, by Lemma \ref{2commutative}, we obtain $ab=ba$ as desired.
    \end{proof}
\end{lemma}
\begin{remark}

  If $a\in\mathcal{A}$ then by Lemma \ref{2commutative}, we have\\ $$\displaystyle\{b:b\in\mathcal{A}\text{, }a,b \text{ are }a\text{-weakly commutative}\}\subset\{b:b\in\mathcal{A}\text{ and }a^2b=ba^2\},$$ but for any non zero complex numbers $\lambda,\lambda^{\prime}$ \begin{multline*}
      \{b:b\in\mathcal{A}\text{, }a,b \text{ are }a\text{-weakly commutative}\}\not\subset \{b:b\in\mathcal{A}\text{ and }\lambda a^2b=\lambda^{\prime}ba^2=aba\}. \end{multline*} This becomes evident in the subsequent example.
\end{remark}
\begin{example}\label{example}
    Let $a=\begin{bmatrix}
        0 & 1\\
        1 & 0 \\
    \end{bmatrix}$ and $b=\begin{bmatrix}
        0 & -1 \\
        0 & 0 \\
    \end{bmatrix}\in\mathcal{A}=M_2(\mathbb{C}).$ Then $a,b$ are $a$-weakly commutative but $a^2b=ba^2\neq aba.$
    \begin{proof}
        Let $c=\begin{bmatrix}
            0 & 0\\
            -1 & 0\\
        \end{bmatrix}$ then $ab=ca\text{ and }ba=ac$ therefore $a,b$ are $a$-weakly commutative. Moreover $a^2b=ba^2=\begin{bmatrix}
            0 & -1\\
            0 & 0\\
        \end{bmatrix}$ but $aba=\begin{bmatrix}
            0 & 0\\
            -1 & 0\\
        \end{bmatrix}.$
    \end{proof}
\end{example}
\begin{remark}
    For $a,b\in\mathcal{A}$, $a$-weakly commutativity of $a,b$ does not implies $b$-weakly commutativity of $a,b$. In Example \ref{example} $a,b$ are $a$-weakly commutative, but they are not $b$-weakly commutative.
\end{remark}
\noindent The aim of this paper is to provide several additive and multiplicative results for Drazin and g-Drazin inverse under this newly defined relation.

The remainder of this paper is presented in the following manner. In Section \ref{addition}, we establish a relation between the Drazin invertibility of $A+B$, $AA^D(A+B)BB^D$, $AA^D(A+B)$, $(A+B)BB^D$ and $A^2A^D(I+A^DB)$ when $A,B\in\mathcal{B}(X)^D$ are $\{A,B\}$-weakly commutative. Furthermore, we provide a representation for $(A+B)^D$ in all the cases mentioned above. In Section \ref{multiplication}, for $A,B\in\mathcal{B}(X)^D$ we investigate the Drazin invertibility of $AB$ when $A,B$ are $\{A,B\}$-weakly commutative. As an application, we establish some new additive and multiplicative results in Section \ref{application}, for some particular types of operators, using the results found in section \ref{addition} and section \ref{multiplication}.  Finally,
in Section \ref{Banach alg}, we extend some of the results presented in section \ref{addition} and section \ref{multiplication} for g-Drazin invertibility within a complex Banach algebra.
\section{Additive properties}\label{addition}
The aim of this section is to establish several results related to the addition of two Drazin invertible operators.
We begin with the following lemma.
\begin{lemma}\label{oldnil}
    If $A,B\in\mathcal{B}(X)^{nil}$, such as $AB=0$, then $A+B$ is nilpotent.
\end{lemma}
\noindent  The following result is proven for Drazin inverse in \cite{hartwig1977group}, and the case of the g-Drazin inverse is established in \cite{gonzalez2004new}.
\begin{theorem}\label{triangular}
    Let $A\in\mathcal{B}(X)$ have the block form representation$$A=\begin{bmatrix}
        A_1 & 0\\
        A_3 & A_2
    \end{bmatrix}.$$ 
    \begin{enumerate}[label=(\roman*)]
        \item If $A_1,A_2$ are Drazin invertible then $A$ is Drazin invertible and
        $$A^D=\begin{bmatrix}
            A_1^D & 0\\
            C & A_2^D\\
        \end{bmatrix},$$ where
        \begin{multline*}
            C=\sum_{n=0}^{i(A_1)-1}(A_2^D)^{n+2}A_3A_1^n(I-A_1A_1^D)+ \sum_{n=0}^{i(A_2)-1}(I-A_2A_2^D)A_2^nA_3(A_1^D)^{n+2}-A_2^DA_3A_1^D.
        \end{multline*}
        \item If $A\text{ and }A_1(\text{or }A_2)$ are Drazin invertible then $A_2(\text{or }A_1)$ is also Drazin invertible. 
    \end{enumerate}

\end{theorem} 
\noindent Now, we present the following beneficial outcome.
\begin{lemma}\label{nil}
If $A,\text{}B\in\mathcal{B}(X)$ are $A$-weakly commutative operators and $A\in\mathcal{B}(X)^{nil}$ then $AB\in\mathcal{B}(X)^{nil}$. Moreover, if $A,B$ are $B$-weakly commutative and $B\in\mathcal{B}(X)^{nil}$, then $A+B\in\mathcal{B}(X)^{nil}$.

\begin{proof}
    Since $A,B$ are $A$-weakly commutative operators then there exists $C_1\in\mathcal{B}(X)$ satisfying
    $AB=C_1A\text{ and } BA=AC_1$. Hence for $k\in\mathbb{N}$ we have
    \begin{align*}
        (AB)^k&=\begin{cases}
            A^kBC_1B\cdots B \text{, if $k$ is odd,}\\A^kC_1BC_1\cdots B \text{, if $k$ is even}.
        \end{cases}
    \end{align*}
    Therefore, if $A$ is nilpotent, then $AB$ is also nilpotent. Moreover $A,B$ are $B$-weakly commutative then $AB=BC_2,\text{ }BA=C_2B$ for some operator $C_2\in\mathcal{B}(X)$. Let $k\geq \text{max}\{k_1,k_2\}$, where $k_1,k_2$ are the nilpotency index of $A$ and $B$, respectively. Then $(A+B)^{3k}$ is the sum of the some monomials of the form 
    \begin{equation}\label{ABAB}
        A^{p_1}B^{q_1}A^{p_2}B^{q_2}\cdots A^{p_m}B^{q_m}, \text{ for }1\leq m\leq 3k
    \end{equation}where $p_i,q_j\in\{0,1,2\cdots 3k\}$ for $1\leq i,j\leq m,$ such that $\displaystyle\sum_{i=1}^{m}(p_i+q_i)=3k.$ In fact in (\ref{ABAB}) at least one of $\displaystyle\sum_{i=1}^{m}p_i$,  $\displaystyle\sum_{i=1}^{m}q_i$ is larger than $k$. Without loss of generality, let us assume $\displaystyle\sum_{i=1}^{m}q_i\geq k.$ Then using $B$-weakly commutativity of $A,B$ we obtain 
    \begin{align*}
         A^{p_1}B^{q_1}A^{p_2}B^{q_2}\cdots A^{p_m}B^{q_m}&=B^kT,\text{ for some }T\in\mathcal{B}(X)\\&=0, \text{ for } 1\leq m\leq 3k.
    \end{align*}
    Hence, each monomial in the expansion of $(A+B)^{3k}$ is $0$. Thus, we conclude that $A+B$ is nilpotent.
\end{proof}
\end{lemma}
\begin{remark}
    In Lemma \ref{nil}, if we take $A^2B=BA^2$ and $AB^2=B^2A$, instead of taking $\{A,B\}$-weakly commutativity, then, in general, $A+B$ may not be nilpotent. This can be observed in the following example.
\end{remark}
\begin{example}Let
    $$A=\begin{bmatrix}
        0 & 0\\
        1 & 0
    \end{bmatrix}\text{, }B=\begin{bmatrix}
        0 & 1\\
        0 & 0
    \end{bmatrix}\in M_2(\mathbb{C}).$$ Then $A,\text{ }B$ are nilpotent and $A^2B=BA^2\text{, }AB^2=B^2A$,  but $A+B=\begin{bmatrix}
        0 & 1\\
        1 & 0
    \end{bmatrix}$, which is not nilpotent.
\end{example}
\noindent The following lemma is important for our main results.
\begin{lemma}\label{new_inv}
    If $A\in \mathcal{B}(X)^{inv}$, $B\in \mathcal{B}(X)^{nil}$ such as $A,B$ are $B$-weakly commutative, and $A^2B=BA^2,$ then $A+B\in\mathcal{B}(X)^{inv}$.
    \begin{proof}
        Let $B\in \mathcal{B}(X)^{nil}$ and $A,B$ are $B$-weakly commutative then by Lemma \ref{nil}, $AB$ is nilpotent. Furthermore since $A\in \mathcal{B}(X)^{inv}$, therefore we have $$A+B=A(I+A^{-1}B).$$ Now $A$ is invertible and $A^2B=BA^2$ implies $A^{-2}(AB)=(AB)A^{-2}$. Then we have
        \begin{multline*}
            (I+A^{-1}B)^{-1} =(I+A^{-2}AB)^{-1}
            =I-A^{-2}AB+(A^{-2})^2(AB)^2+\cdots +(-1)^{m-1}(A^{-2})^{m-1}(AB)^{m-1},
        \end{multline*}
      where $m\text{ is the nilpotency index of }AB$. Hence we obtain  $A+B=A(I+A^{-1}B)$ is invertible.    
    \end{proof}
\end{lemma}
 \begin{remark}
     
 In Lemma \ref{new_inv}, if $A,\text{}B$ is not $B$-weakly commutative, then $A+B$ may not be invertible; the next example describes that.
 \end{remark} 
  \begin{example}
     Let $A=\begin{bmatrix}
         0 & 1\\
         1 & 0
     \end{bmatrix}$ and $B=\begin{bmatrix}
         0 & -1\\
         0 & 0
     \end{bmatrix}\in M_2(\mathbb{C})$. Here $A,\text{ }B$ is not $B$-weakly commutative but $A^2B=BA^2.$ Moreover $A$ is invertible, $B$ is nilpotent, but $A+B$ is not invertible. 
 \end{example}
\begin{corollary}\label{inv}
    If $A\in \mathcal{B}(X)^{inv}$, $B\in \mathcal{B}(X)^{nil}$ such as $A,B$ are $\{A,B\}$-weakly commutative. Then $A+B\in\mathcal{B}(X)^{inv}$.
    \begin{proof}
        If $A,B$ are $A$-weakly commutative, then by Lemma \ref{2commutative}, $A^2B=BA^2.$ Therefore the required result follows from Lemma \ref{new_inv}.
    \end{proof}
\end{corollary}
\noindent Next, we demonstrate one of the main theorems of this section. 
\begin{theorem}\label{Drazin_main}
    If $A,\text{}B\in\mathcal{B}(X)^D$ such as $A,B$ are $\{A,B\}$-weakly commutative, then the followings are equivalent:
    \begin{enumerate}[label=(\roman*)]
        \item $A+B\in \mathcal{B}(X)^D;$ \label{con_1}
        \item $AA^D(A+B)BB^D\in \mathcal{B}(X)^D;$\label{con_2}
        \item $AA^D(A+B)\in \mathcal{B}(X)^D;$\label{con_3}
        \item $(A+B)BB^D \in \mathcal{B}(X)^D.$\label{con_4}
    \end{enumerate}
    Before proceeding with the proof of Theorem \ref{Drazin_main}, we must establish the following result, which holds significance in the context of this theorem:

\begin{theorem}\label{Dra_Nil}
    Let $A\in \mathcal{B}(X)^D$, $B\in\mathcal{B}(X)^{nil}$ such as  $A,B$ are $\{A,B\}$-weakly commutative. Then $A+B\in\mathcal{B}(X)^D$ and \begin{align*}
                (A+B)^D&=(I+A^DB)^{-1}A^D
                \\&=\left(\sum_{n=0}^{i(A)-1}(A^D)^{2n}(-AB)^n\right)A^D.
            \end{align*}
    \begin{proof}
    
       Since $A\in\mathcal{B}(X)^D$, then $X$  has the decomposition $X=\mathcal{R}(AA^D)\oplus \mathcal{N}(AA^D)$ and corresponding to this decompositon of $X$, $A$ have the block form $$A=A_1\oplus A_2,$$ where $A_1$ is invertible operator on $\mathcal{R}(AA^D)$ and $A_2$ is nilpotent operator on $\mathcal{N}(AA^D).$ Moreover $A,B$ are $\{A,B\}$-weakly commutative means $AB=C_1A,\text{ }BA=AC_1$ and $AB=BC_2,\text{ }BA=C_2B$ for two operators $C_1,\text{ }C_2\in\mathcal{B}(X)$. Since $AB=C_1A,\text{ }BA=AC_1$, then both the spaces $\mathcal{R}(AA^D)\text{ and } \mathcal{N}(AA^D)$ are invariant under $B$ and $C_1.$  Therefore relative to the decomposition $X=\mathcal{R}(AA^D)\oplus \mathcal{N}(AA^D)$, $B$ have the block representation $$B=B_1\oplus B_2,$$ where both $B_1\text{ and }B_2$ are nilpotent operators on $\mathcal{R}(AA^D)$ and $\mathcal{N}(AA^D)$, respectively. Similarly we have $$C_1=C_{11}\oplus C_{12},$$ where $C_{11}\text{ , }C_{12}$ are two operators on $\mathcal{R}(AA^D)$ and $\mathcal{N}(AA^D),$ respectively. Using this block forms of $A,B\text{ and }C_1$ in $AB=C_1A,\text{ }BA=AC_1$ we obtain 
        \begin{align*}
             A_1B_1=C_{11}A_1,\text{ }B_1
            A_1=A_1C_{11}\text{ and } A_2B_2=C_{12}A_2,\text{ }B_2A_2=A_2C_{12}.
        \end{align*}
 Again using this block forms in $AB=BC_2, \text{ } BA=C_2B$ we get $$ A_1B_1=B_1C_{21},\text{ }B_1
            A_1=C_{21}B_1\text{ and } A_2B_2=B_2C_{22},\text{ }B_2A_2=C_{22}B_2,$$ where  $\begin{bmatrix}
            C_{21} & C_{23}\\
            C_{24} & C_{22}
            \end{bmatrix}
            $ is the block representation of $C_2$ corresponding to the space decomposition $X=\mathcal{R}(AA^D)\oplus \mathcal{N}(AA^D).$ Hence we get that $A_1,B_1$ are $\{A_1,B_1\}$-weakly commutative and $A_2,B_2$ are $\{A_2,B_2\}$-weakly commutative. Therefore, by Lemma \ref{nil}, $A_2+B_2$ is nilpotent, and $A_1+B_1$ is invertible by Corollary \ref{inv}. Hence we obtain $A+B=(A_1+B_1)\oplus (A_2+B_2)\in\mathcal{B}(X)^D$ by Theorem \ref{triangular}, and 
            \begin{align*}
                (A+B)^D&=(A_1+B_1)^{-1}\oplus 0
                \\&=(I+A^DB)^{-1}A^D
                \\&=\left(\sum_{n=0}^{i(A)-1}(A^D)^{2n}(-AB)^n\right)A^D,
            \end{align*}
            as desired.
            \end{proof}
\end{theorem}
    \begin{proof}[Proof of Theorem \ref{Drazin_main}]
        Since $A,B$ are $\{A,B\}$-weakly commutative, then similar to the proof of Theorem \ref{Dra_Nil}, we get $$A=A_1\oplus A_2\text{ and }B=B_1\oplus B_2,$$ where $A_1$ is an invertible operator on $\mathcal{R}(AA^D)$, $A_2$ is an nilpotent operator on $\mathcal{N}(AA^D)$ and $B_1,B_2$ are two Drazin invertible operators on $\mathcal{R}(AA^D) $ and $ \mathcal{N}(AA^D)$, respectively. Moreover, similar to the proof of Theorem \ref{Dra_Nil} we also obtain $A_1,B_1$ are $\{A_1,B_1\}$-weakly commutative and $A_2,B_2$ are $\{A_2,B_2\}$-weakly commutative. Since $A_2$ is nilpotent, and $B_2$ is Drazin invertible therefore, by Theorem \ref{Dra_Nil}, $A_2+B_2$ is Drazin invertible. Thus $A+B=(A_1+ B_1)\oplus (A_2+B_2)\in\mathcal{B}(X)^D$ if and only if $A_1+B_1$ is Drazin invertible, hence \ref{con_1}$\iff $\ref{con_3}. Again since $A_1,\text{ }B_1$ are $A_1\text{ and }B_1$-weakly commutative, then corresponding to the decomposition $$\mathcal{R}(AA^D)=\mathcal{R}(B_1B_1^D)\oplus \mathcal{N}(B_1B_1^D),$$ $B_1\text{ and }A_1$ has the block representation $B_1=B_{11}\oplus B_{12}$ and $A_1=A_{11}\oplus A_{12}$, respectively. Furthermore $B_{11},\text{ }A_{11}$ are $\{A_{11},B_{11}\}$-weakly commutative and $B_{12},\text{}A_{12}$ are $\{B_{12},A_{12}\}$-weakly commutative. Since $A_{12}$ is invertible and $B_{12}$ is nilpotent, then by Corollary \ref{inv}, we get $A_{12}+B_{12}$ is invertible. Hence by Theorem \ref{triangular}, $A_1+B_1=(A_{11}+B_{11})\oplus (A_{12}+B_{12})$ is Drazin invertible if and only if $A_{11}+B_{11}$ is Drazin invertible, thus we obtain \ref{con_1}$\iff$ \ref{con_2}, \ref{con_1}$\iff$ \ref{con_4}.
    \end{proof}
\end{theorem} 
\noindent Now, we present the following finding,  which includes an  expression for $(A+B)^D$ as well.
\begin{theorem}
    Let $A,B\in\mathcal{B}(X)^D$, such as $A,B$ are $\{A,B\}$-weakly commutative and $A^2A^D(I+A^DB)\in\mathcal{B}(X)^D$. Then $A+B\in \mathcal{B}(X)^D$, and
    \begin{align*}
          (A+B)^D&{}=\left(A^2A^D(I+A^DB)\right)^D+(I-AA^D)B^D\left(A(I-AA^D)B^D+I\right)^{-1}
        \\ {}&=\left(A^2A^D(I+A^DB)\right)^D+(I-AA^D)B^D\left(\sum_{n=0}^{i(A)-1}(-AB)^{n-1}(B^D)^{2n}\right).
    \end{align*}
    \begin{proof}
        Since here $A,B$ are $\{A,B\}$-weakly commutative, therefore as in the proof of Theorem \ref{Dra_Nil}, we have $$A=A_1\oplus A_2 \text{ and }B=B_1\oplus B_2,$$ where $A_1,B_1$ are $\{A_1,B_1\}$-weakly commutative and  $A_2,B_2$ are $\{A_2$,$B_2\}$-weakly commutative. Here $A^2A^D(I+A^DB)\in\mathcal{B}(X)^D$, therefore  $(A_1+B_1)\in\mathcal{B}(\mathcal{R}(AA^D))^D$ and by Theorem \ref{triangular}, we have 
        \begin{equation}\label{101}
            (A_1+B_1)^D\oplus 0=\left(A^2A^D(I+A^DB)\right)^D.
        \end{equation}
         Furthermore $A_2,B_2$ are $\{A_2,B_2\}$-weakly commutative and $A_2$ is nilpotent, $B_2$ is Drazin invertible in $\mathcal{N}(AA^D).$ Therefore by Theorem \ref{Dra_Nil}, $A_2+B_2\in(\mathcal{N}(AA^D))^D$ and 
        \begin{align}
            0\oplus (A_2+B_2)^D&=(I-AA^D)B^D\left(A(I-AA^D)B^D+I\right)^{-1}\label{102} \\&=(I-AA^D)B^D\left(\sum_{n=0}^{i(A)-1}(-AB)^{n-1}(B^D)^{2n}\right).\label{103}
        \end{align}
       Hence the required result follows from Theorem \ref{triangular}, equation (\ref{101}), (\ref{102}) and (\ref{103}).  
    \end{proof}
\end{theorem}
Similar to the previous result, using Theorem \ref{Drazin_main} and Theorem \ref{Dra_Nil} one can establish the following corollary, which provides representations for $(A+B)^D$.
\begin{corollary}
Let $A,B\in\mathcal{B}(X)^D,$ such as $A.B$ are $\{A,B\}$-weakly commutative.  
\begin{enumerate}[label=(\roman*)]

    \item If $AA^D(A+B)BB^D\in\mathcal{B}(X)^D$ then 
    $$\begin{aligned}[t]
        (A+B)^D={}&\left(AA^D(A+B)BB^D\right)^D+(I-AA^D)B^D\left(\sum_{n=0}^{i(A)-1}(-AB)^n(B^D)^{2n}\right)\\{}&+\left(\sum_{n=0}^{i(B)-1}(A^D)^{2n}(-AB)^n\right)A^D(I-BB^D)
    \end{aligned}$$
    \item If $AA^D(A+B)\in\mathcal{B}(X)^D,$ then $$(A+B)^D=\left(AA^D(A+B)\right)^D+(I-AA^D)B^D\left(\sum_{n=0}^{i(A)-1}(-AB)^n(B^D)^{2n}\right),$$
    \item If $(A+B)BB^D\in\mathcal{B}(X)^D$, then $$(A+B)^D=\left((A+B)BB^D\right)^D+\left(\sum_{n=0}^{i(B)-1}(A^D)^{2n}(-AB)^n\right)A^D(I-BB^D).$$

\end{enumerate}
\end{corollary}

\noindent The subsequent outcome presents a broader perspective of Theorem \ref{Dra_Nil}.
\begin{theorem}
    Let $A\in\mathcal{B}(X)^D\text{, }B\in\mathcal{B}(X)^{nil}$ such as $A,B$ are $B$-weakly commutative. If $AA^D(A^2B-BA^2)=0$, $(A^2B-BA^2)AA^D=0$, $A^{\pi}(AB-CA)=0$ and $A^{\pi}(BA-AC)=0$ for some operator $C\in\mathcal{B}(X)$, then $A+B\in\mathcal{B}(X)^D.$
    \begin{proof}
    Since $A\in\mathcal{B}(X)^D$, therefore $A$ have the block representation $$A=A_1\oplus A_2,$$
         where $A_1\in\mathcal{B}(\mathcal{R}(AA^D))^{inv},$ and $A_2\in \mathcal{B}(\mathcal{N}(AA^D))^{nil}.$ Again since $AA^D(A^2B-BA^2)=0$ and $(A^2B-BA^2)AA^D=0$ then $B$ also have the block representation $$B=B_1\oplus B_2,$$ where $B_1,B_2$ both are nilpotent operators on $\mathcal{R}(AA^D)$ and $\mathcal{N}(AA^D),$ respectvely. Moreover $A,B$ are $B$-weakly commutative then $A_1,B_1$ are $B_1$-weakly commutative and $A_2,B_2$ are $B_2$-weakly commutative. Furthermore using these block representations of $A$ and $B$ on the condition $A^{\pi}(AB-CA)=0$ and $A^{\pi}(BA-AC)=0$ we obtain $A_2,B_2$ are $A_2$-weakly commutative also. Hence we have $A_2,B_2$ are $\{A_2,B_2\}$-weakly commutative, therefore by Lemma \ref{nil}, $A_2+B_2\in\mathcal{B}(\mathcal{N}(BB^D))^{nil}.$ Again from $AA^D(A^2B-BA^2)=0$ we obtain $A_1^2B_1=B_1A_1^2$, therefore by Lemma \ref{new_inv}, $A_1+B_1\in\mathcal{B}(\mathcal{R}(AA^D))^{inv}.$ Hence using Theorem \ref{triangular}, we obtain $$A+B=(A_1+B_1)\oplus (A_2+B_2)\in\mathcal{B}(X)^D,$$ as desired.
    \end{proof}
\end{theorem}
\noindent DS Djordjević and Y Wei \cite{djordjevic2002additive} established  that if $A,B\in\mathcal{B}(X)^d$ and $AB=0$ then $A+B\in\mathcal{B}(X)^d.$ This same conclusion holds in situations involving the Drazin inverse also. Consequently, utilizing this finding leads us to the following theorem. 
\begin{theorem}
    Let $A,B\in\mathcal{B}(X)^D$ such as $A,B$ are $A$-weakly commutative, and $A^{\pi}A,A^{\pi}B$ are $A^{\pi}B$-weakly commutative. If $AA^DAB=0$, then $A+B\in\mathcal{B}(X)^D.$
    \begin{proof}
     Similar to the proof of Theorem \ref{Dra_Nil}, we have $$A=A_1\oplus A_2\text{ and }B=B_1\oplus B_2,$$ where $A_1\in\mathcal{B}(\mathcal{R}(AA^D))^{inv}$, $B_1\in \mathcal{B}(\mathcal{R}(AA^D))^D$, $A_2\in\mathcal{B}(\mathcal{N}(AA^D))^{nil}$ and $B_2\in\mathcal{B}(\mathcal{N}(AA^D))^D.$ It follows from the $A$-weakly commutativity of $A,B$ that $A_1,B_1$ are $A_1$-weakly commutative and $A_2,B_2$ are $A_2$-weakly commutative. Furthermore, from $A^{\pi}B$-weakly commutativity of $A^{\pi}A,A^{\pi}B$ we obtain $A_2,B_2$ are $B_2$-weakly commutative. Therefore using Theorem \ref{Dra_Nil}, we get $A_2+B_2\in\mathcal{B}(\mathcal{N}(AA^D))^D.$ Moreover $AA^DAB=0$ implies $A_1B_1=0.$ Therefore $A_1+B_1\in\mathcal{B}(\mathcal{R}(AA^D))^D.$ Hence using Theorem \ref{triangular}, we get
         $A+B=(A_1+B_1)\oplus (A_2+B_2)\in\mathcal{B}(X)^D$.
    \end{proof}
\end{theorem}
\section{Multiplicative properties}\label{multiplication}
In this section, we investigate the Drazin invertibility of $AB$ for a pair of Drazin invertible operators $A\text{ and }B$, where $A,\text{ }B$ are $\{A,B\}$-weakly commutative. We start by presenting some results that hold significance for the main theorem of this section.
\begin{theorem}\label{idemp}
If $A\in\mathcal{B}(X)^D$ such that  $AB=CA$ and $BA=AC$ for some $B,C\in\mathcal{B}(X)$. Then
\begin{enumerate}[label=(\roman*)]
    \item $AA^DB=BAA^D;$\label{1111}
    \item $AA^DC=CAA^D;$
    \item $A^DB=CA^D;$\label{3}
    \item $A^DC=BA^D.$
\end{enumerate}
\begin{proof}
\begin{enumerate}[label=(\roman*)]
    \item 
Let $A\in \mathcal{B}(X)^D$, then $A^k(I-AA^D)=0$ for $k\geq i(A).$ Since $AB=CA$ and $BA=AC$ implies $A^2B=BA^2$,  then we have
\begin{align*}
    AA^DB-AA^DBAA^D&=AA^DB(I-AA^D)\\&=
(A^D)^{2k}A^{2k}B(I-AA^D)\\&=(A^D)^{2k}
BA^{2k}(
I-AA^D)=0.
\end{align*}
Therefore $AA^DB=AA^DBAA^D,$ similarly we get $BAA^D=AA^DBAA^D.$ Hence $AA^DB=BAA^D.$
\item Since $AB=CA$ and $BA=AC$ implies $A^2C=CA^2.$ Therefore this is similar to \ref{1111}.

\item Now
\begin{align*}
    A^DB=(A^D)^2AB&=A^DBAA^D\\&=A^DACA^D\\&=CA^D.
\end{align*}
\item Follows similarly as \ref{3}.
\end{enumerate}
\end{proof}
\end{theorem}
\noindent The subsequent corollary follows from Theorem \ref{idemp}.
\begin{corollary}\label{block_coro}
    Let $A,B\in\mathcal{B}(X)$ are $A$-weakly commutative operators and $A$ is Drazin invertible then $AA^DB=BAA^D.$
\end{corollary}

\noindent Now, we are ready to look over the Drazin invertibility of  product of two Drazin invertible operators.
\begin{theorem}\label{product}
    Let $A,B\in\mathcal{B}(X)^D$. If $A,B$ are $\{A,B\}$-weakly commutative, then $AB$ is Drazin invertible and $(AB)^D=B^DA^D.$
    \begin{proof}
    Since $A,B$ are $\{A,B\}$-weakly commutative, then $$AB=C_1A,\text{ }BA=AC_1 \text{ and } AB=BC_2,\text{ }BA=C_2B,$$ for two operators $C_1,C_2\in\mathcal{B}(X).$ 
        Let $X=B^DA^D$. Then, using Corollary \ref{block_coro}, we can ascertain that $ABX=XAB$ and $XABX=X$. Moreover we have \begin{align*}
                AB-ABXAB&=AB-ABB^DA^DAB\\&=AB-A^2A^DB^DB^2\\&=A^2A^D(B-B^2B^D)+(A-A^2A^D)B.
            \end{align*}
            Since $A-A^2A^D$ is nilpotent and $$(A-A^2A^D)B=C_1(A-A^2A^D),\text{ }B(A-A^2A^D)=(A-A^2A^D)C_1$$ therefore by Lemma \ref{nil}, $(A-A^2A^D)B$ is nilpotent. Similarly since $B-B^2B^D$is nilpotent and $$
       A^2A^D(B-B^2B^D)=(B-B^2B^D)C_2AA^D,\text{ }(B-B^2B^D)A^2A^D=C_2AA^D(B-B^2B^D)$$ therefore $A^2A^D(B-B^2B^D)$ is also nilpotent. Furthermore, we have $$
       (A-A^2A^D)BA^2A^D(B-B^2B^D)=0 \text{ and }(A-A^2A^D)A^DABA(B-B^2B^D)=0.
       $$ Therefore by Lemma \ref{oldnil}, we get $AB-ABB^DA^DAB$ is nilpotent.
    \end{proof}
\end{theorem}
\begin{example}
    Let $\mathcal{B}(X)=\mathcal{M}_3(\mathbb{C})$,   $$A=\begin{bmatrix}
        0 & x & 0 \\
        x & 0 & x \\
        0 & x & 0 
    \end{bmatrix} \text{ and } B=\begin{bmatrix}
        y & 0 & 0 \\
        0 & -y & 0\\
        0 & 0 & y
    \end{bmatrix}\in M_3(\mathbb{C}), $$ where $x,y\in\mathbb{C}-\{0\}.$ Therefore $A,B$ are $\{A,B\}$-weakly commutative, and hence $(AB)^D=B^DA^D$.
    \begin{proof}
        Let $C_1=-B$ and $C_2=-A$ then we have $AB=C_1A,$ $BA=AC_1$ and $AB=BC_2$, $BA=C_2B.$ Moreover we have $$A^D=\begin{bmatrix}
        0 & \frac{1}{2x} & 0 \\
        \frac{1}{2x} & 0 & \frac{1}{2x} \\
        0 & \frac{1}{2x} & 0 
    \end{bmatrix} \text{, } B^D=\begin{bmatrix}
        \frac{1}{y} & 0 & 0 \\
        0 & -\frac{1}{y} & 0\\
        0 & 0 & \frac{1}{y}
    \end{bmatrix}\text{ and }(AB)^D=\begin{bmatrix}
        0 & \frac{1}{2xy} & 0 \\
        -\frac{1}{2xy} & 0 & -\frac{1}{2xy} \\
        0 & \frac{1}{2xy} & 0 
    \end{bmatrix}.$$ Hence we obtain $(AB)^D=B^DA^D$.
    \end{proof}
\end{example}
\begin{remark}
    The results found in \cite[Theorem 7.8.4]{campbell1991generalized} followed as a specific case of Theorem \ref{product} and Theorem \ref{idemp}.
\end{remark}
\noindent Utilizing the preceding result, we obtain the subsequent corollary. 

\begin{corollary}
    If $A,B\in\mathcal{B}(X)^D$ such as $A,B$ are $A$-weakly commutative and $AA^DA,AA^DB$ are $AA^DB$-weakly commutative. Then $AB\in\mathcal{B}(X)^D$ and $(AB)^D=B^DA^D.$ 
    \begin{proof}
        Let $A\in\mathcal{B}(X)^D$ and $A,B$ are $A$-weakly commutative therefore $A\text{ and }B$ has the block form representation $$A=A_1\oplus A_2\text{, }B=B_1\oplus B_2,$$ where $A\in\mathcal{B}(\mathcal{R}(AA^D))^{inv}$ and $A_2\in\mathcal{B}(\mathcal{N}(AA^D))^{nil}$. Moreover from $A$-weakly commutativity of $A,B$ it follows that $A_1,B_1$ are $A_1$-weakly commutative and $A_2,B_2$ are $A_2$ weakly commutative. Then, by Lemma \ref{nil}, $A_2B_2$ is nilpotent since $A_2$ is nilpotent. Furthermore from the $AA^DB$-weakly commutativity of $AA^DA$ and $AA^DB$ we obtain $B_1$-weakly commutativity of $A_1,B_1$. Therefore $A_1,B_1$ are $\{A_1,B_1\}$-weakly commutative, then by Theorem \ref{product}, $A_1B_1\in\mathcal{B}(\mathcal{R}(AA^D))$ and $(A_1B_1)^D=B_1^DA_1^D.$ Hence using Theorem \ref{triangular}, we get $$AB=A_1B_1\oplus A_2B_2\in\mathcal{B}(X)^D,\text{ and }(AB)^D=B^DA^D.$$
    \end{proof}
\end{corollary}
\section{Applications}\label{application}
The aim of this section is to explore new additive and multiplicative properties of Drazin inverse for some specific types of operators by using results found in section \ref{addition} and section \ref{multiplication}.
 We begin with the subsequent lemma.
\begin{lemma}
    If $A,B,C\in\mathcal{B}(X),$ such as $AB=CA,$ $BA=AC$ and $A$ is invertible. Then $A+B\in\mathcal{B}(X)^D$ if and only if $A+C\in\mathcal{B}(X)^D$.
    \begin{proof}
        Since $A$ is invertible then $A^D=A^{-1}$ and from $AB=CA$ we obtain  $B=A^{-1}CA$. Moreover we have $$A+B=A+A^{-1}CA=(I+A^{-1}C)A.$$ Therefore by Cline’s formula \cite{cline1965application}, $(I+A^{-1}C)A\in\mathcal{B}(X)^D$ if and only if $A(I+A^{-1}C)=A+C$ is Drazin invertible. Hence, we get $A+B\in\mathcal{B}(X)^D$ if and only if $A+C\in\mathcal{B}(X)^D$. 
    \end{proof}
\end{lemma}
By applying Theorem \ref{product}, we obtain the following result, which provides a sufficient condition for the Drazin invertibility of the sum of two involutary operators.
\begin{theorem}\label{involutary}
 Let $A,B\in\mathcal{B}(X)$ such that $A^2=B^2=I$ and $A-B=BAB-ABA$. If $I+AB\in\mathcal{B}(X)^D$ then $A+B\in\mathcal{B}(X)^D$, and $$(A+B)^D=(I+AB)^DA.$$
    \begin{proof}
       If $A^2=B^2=I$ then we have $A+B=A(I+AB).$  Moreover
           $$A(I+AB)=(I+BA)A\text{ and }
           (I+AB)A=A(I+BA).$$
        Therefore $A, I+AB$ are $A$-weakly commutative. Furthermore, we have $$A(I+AB)=(I+AB)B$$ and from $A-B=BAB-ABA$ it follows that $(I+AB)A=B(I+AB)$, hence $A, I+AB$ are $I+AB$-weakly commutative. Since $A$ is invertible, therefore, by Theorem \ref{product}, if $I+AB\in\mathcal{B}(X)^D$, then $A(I+AB)=A+B$ is Drazin invertible where \begin{align*}
           (A+B)^D=&(I+AB)^DA^D\\=&(I+AB)^DA.
       \end{align*}
    \end{proof}
\end{theorem}
\newpage
Now, we give an example to illustrate the preceding theorem.
\begin{example}
    Let $T,S\in\mathcal{B}(\ell^2(\mathbb{C})),$ define by
    \begin{align*}
        T(x_1,x_2,x_3,x_4,x_5,x_6,\cdots)={}&(-x_1,x_2-10x_3,-x_3,-x_4,x_5-10x_6,-x_6,\cdots)\text{ and }\\S(x_1,x_2,x_3,x_4,x_5,x_6,\cdots)={}&(-x_1+2x_3,-x_2-10x_3,x_3,-x_4+2x_6,-x_5-10x_6,x_6,\cdots).
    \end{align*}Then one can easily verify that $T^2=I$, $S^2=I$ and $T-S=STS-TST.$ Therefore by Theorem \ref{involutary}, if $I+TS$ is Drazin invertible then $T+S$ is Drazin invertible. 
\end{example}
Next, using Lemma \ref{nil}, we prove that if in Definition \ref{defn_weak}, we restrict the choice of $B$ and $C$ to the collection of invertible elements only, then the Drazin invertibility of $AA^D(A+B)$ ensure the Drazin invertibility of $A+B$, when $A,B$ are $A$-weakly commutative.
\begin{theorem}
    Let $A\in\mathcal{B}(X)^D$ and $B,C\in\mathcal{B}(X)^{inv}$ such as $AB=CA$ and $BA=AC$. If $AA^D(A+B)\in \mathcal{B}(X)^D,$ then $A+B\in\mathcal{B}(X)^D$.
    \begin{proof}
        Since $A\in\mathcal{B}(X)^D$, therefore $A$ have the block representation $$A=A_1\oplus A_2,$$
         where $A_1\in\mathcal{B}(\mathcal{R}(AA^D))^{inv},$ and $A_2\in \mathcal{B}(\mathcal{N}(AA^D))^{nil}.$ Furthermore since $AB=CA$, $BA=AC$ then $B$ also have the block representation $$B=B_1\oplus B_2,$$ where $B_1$
        and $B_2$ are invertible operators in $\mathcal{B}(\mathcal{R}(AA^D))$ and $\mathcal{B}(\mathcal{N}(AA^D)) $, respectively. Moreover from $AA^D(A+B)\in \mathcal{B}(X)^D$ we obtain that $A_1+B_1\in\mathcal{B}(\mathcal{R}(AA^D))^D$. Now since $B,C$ are invertible, therefore $AB=CA$ and $BA=AC$ implies $$AB^{-1}=C^{-1}A\text{ and }B^{-1}A=AC^{-1}.$$ Hence $A,B^{-1}$ are $A$-weakly commutative, moreover $A_2,B_2^{-1}$ are also $A_2$-weakly commutative. Now since $A_2\in \mathcal{B}(\mathcal{N}(AA^D))^{nil},$ then by Lemma \ref{nil}, $A_2B_2^{-1}\in \mathcal{B}(\mathcal{N}(AA^D))^{nil}$. Therefore $I+A_2B_2^{-1}\in\mathcal{B}(\mathcal{N}(AA^D))^{inv},$ hence $A_2+B_2=(A_2B_2^{-1}+I)B_2\in\mathcal{B}(\mathcal{N}(AA^D))^{inv}.$ Therefore by Theorem \ref{triangular}, $A+B\in\mathcal{B}(X)^D,$ since $A_1+B_1\in\mathcal{B}(\mathcal{R}(AA^D))^D$ and $A_2+B_2\in\mathcal{B}(\mathcal{N}(AA^D))^{inv}.$ 
         \end{proof}
\end{theorem}
\begin{proposition}
 If $A\in(\mathcal{B}(X))^D,$ such that $AB=CA$ and $BA=AC$ for two commuting operators $B,\text{}C\in\mathcal{B}(X)^{inv}$, then $A^2B=ACA$ and $ABA=A^2C$ are Drazin invertible.
    \begin{proof}
        Let $M=AB=CA$ and $N=BA=AC$, then by \cite[Proposition 2.5.]{barraa2019drazin} both $M\text{ and }N$ are Drazin invertible. Furthermore we have \begin{equation*}
             AM=(AC)A=M(B^{-1}M) \text{ and } MA=A(BA)=(MC^{-1})M,
        \end{equation*} where $B^{-1}M=MC^{-1}$ since $BC=CB.$ Therefore $A,M$ are $\{A,M\}$-weakly commutative. Similarly, $A,N$ are also  $\{A,N\}$-weakly commutative. Therefore, in light of Theorem \ref{product}, we obtain both $AN\text{ and }AM$ are Drazin invertible, as desired.    
    \end{proof}
\end{proposition}
\section{For g-Drazin inverse in a Banach algebra}\label{Banach alg}
In this section, we extend the scope of certain additive and multiplicative results presented in Section \ref{addition} and \ref{multiplication} to the case of g-Drazin inverse within a complex Banach algebra.
Let $\mathcal{A}$ be an unital complex Banach algebra with the unit $1$. If $p\in\mathcal{A}$ is an idempotent, then each $x\in\mathcal{A}$ has a unique matrix representation$$\begin{bmatrix}
    x_{11} & x_{12}\\
    x_{21} & x_{22}\\
\end{bmatrix}_p, \text{ where }x_{ij}=p_ixp_j,\text{ }p_1=p,\text{ }p_2=1-p.$$ Now, we begin with the following result presented in \cite{cvetkovic2006additive}. 
\begin{lemma}\cite{cvetkovic2006additive}\label{basicqnil}
    If $a,b\in\mathcal{A}^{qnil}$ such as $ab=0$ or $ab=ba$, then $a+b\in\mathcal{A}^{qnil}.$
\end{lemma}
\begin{lemma}\label{bcommu}
  If $a,b\in\mathcal{A}$ are $a$-weakly commutative such that $a\in\mathcal{A}^{qnil}$ then $ab\in\mathcal{A}^{qnil}.$ Moreover, if $a,b$ are $b$-weakly commutative and $b\in\mathcal{A}^{qnil}$, then $a+b\in\mathcal{A}^{qnil}.$
  \begin{proof}
      Since $a,b$ are $a$-weakly commutative then $ab=c_1a\text{ and }ba=ac_1$ for some $c_1\in\mathcal{A}.$ Let $\alpha=\text{max}\{1,\|a\|+\|b\|+\|c_1\|\},$ and take $a_1=\frac{a}{\alpha},\text{ }b_1=\frac{b}{\alpha}\text{ and }c_1'=\frac{c_1}{\alpha},$ then max$\{\|a_1\|,\|b_1\|,\|c_1'\|\}\leq 1$, furthermore $ab\in\mathcal{A}^{qnil}$ if and only if $a_1b_1\in\mathcal{A}^{qnil}.$ Thus without loss of generality we can assume that max$\{\|a\|,\|b\|,\|c_1\|\}\leq 1$. Now for $n\in\mathbb{N}$ we have
     \begin{align*}
        (ab)^n&=\begin{cases}
          a^nbc_1b\cdots b, \text{ if $n$ is odd,}\\a^nc_1bc_1\cdots b, \text{ if $n$ is even}.  
        \end{cases} 
    \end{align*}
   Therefore $$\lim_{n\to\infty}\|(ab)^n\|^{\frac{1}{n}}\leq\lim_{n\to\infty}\|a^n\|^{\frac{1}{n}}=0,$$ hence we get $ab\in\mathcal{A}^{qnil}.$  In order to prove $a+b\in\mathcal{A}^{qnil}$, its enough to prove that $(a+b)^2\in\mathcal{A}^{qnil}.$ Now $$(a+b)^2=a^2+ab+ba+b^2,$$ where $a^2,b^2,ab,ba\in\mathcal{A}^{qnil}$, since $b\in\mathcal{A}^{qnil}$. Moreover, if $a,b$ are $\{a,b\}$-weakly commutative, then by Lemma \ref{2commutative}, $a^2,ab,ba,b^2$ all commutes with each other. Therefore in light of Lemma \ref{basicqnil}, we obtain $(a+b)^2\in\mathcal{A}^{qnil}.$
  \end{proof}
\end{lemma}
\begin{theorem}\label{alg_prod}
    Let $a\in\mathcal{A}^d$ and $ab=ca\text{, }ba=ac$ for some $b,c\in\mathcal{A}$. Then the following are true:
    \begin{enumerate}[label=(\roman*)]
        \item\label{f1} $aa^db=baa^d$;
        \item \label{f2} $aa^dc=caa^d$;
        \item \label{f3} $a^db=ca^d$;
        \item \label{f4} $a^dc=ba^d$.
    \end{enumerate}
    \begin{proof}
    Since $a\in\mathcal{A}^d$, therefore $a(1-aa^d)\in\mathcal{A}^{qnil},$ therefore $\displaystyle\lim_{n\to\infty}\|\left(a-a^2a^d\right)^n\|^{\frac{1}{n}}=0.$
    \begin{enumerate}[label=(\roman*)]
        \item Now we have \begin{align*}
            aa^db-aa^dbaa^d&=aa^db(1-aa^d)\\&=(a^d)^{2n}a^{2n}b(1-aa^d)\\&=(a^d)^{2n}ba^{2n}(1-aa^d)\\&=(a^d)^{2n}b(a-a^2a^d)^{2n}. \end{align*}
        Hence we obtain $$\|aa^d-aa^dbaa^d\|^{\frac{1}{2n}}\leq \|a^d\|\|b\|^{\frac{1}{2n}}\|(a-a^2a^d)^{{2n}}\|^{\frac{1}{2n}}.$$ Therefore $$\lim_{n\to\infty}\|aa^d-aa^dbaa^d\|^{\frac{1}{2n}}=0,$$ and we get $aa^db-aa^dbaa^d=0.$ Similarly, $baa^d-aa^dbaa^d=0$, thus we have $aa^db=baa^d,$ as required.
        \item Similar like \ref{f1}, follows from the fact that $a^2c=ca^2.$
    \end{enumerate}
     \hspace{0.2cm}   \ref{f3} and \ref{f4} follows similarly as Theorem \ref{idemp}.
    \end{proof}
\end{theorem} 
\noindent The previous lemma yields the following result.
\begin{corollary}
   Let $a,b\in\mathcal{A}$ are $a$-weakly commutative. If $a$ is g-Drazin invertible then $aa^db=baa^d.$
\end{corollary} 
\noindent In accordance with Theorem \ref{product} using Theorem \ref{alg_prod} and Lemma \ref{bcommu}, we obtain the following result.
\begin{theorem}
    If $a,b\in\mathcal{A}^d$, are $\{a,b\}$-weakly commutative then $ab\in\mathcal{A}^d$ and $(ab)^d=b^da^d.$
\end{theorem}
\begin{theorem}\label{inv_nil}
    If $a\in\mathcal{A}^{inv}$ and $b\in\mathcal{A}^{qnil}$ such as $a,b$ are $\{a,b\}$-weakly commutative, then $a+b\in\mathcal{A}^{inv}.$
    \begin{proof}
        Since $a,b$ are $\{a,b\}$-weakly commutative, then by Lemma \ref{2commutative}, we have $a^2b=ba^2$ and by Lemma \ref{bcommu}, $ab\in\mathcal{A}^{qnil}.$ Moreover $$a+b=a(1+a^{-1}b)=a(1+a^{-2}ab).$$ Now since $a\in\mathcal{A}^{inv}$, then $a^2b=ba^2$ implies $a^{-2}(ab)=(ab)a^{-2}$ but $ab\in\mathcal{A}^{qnil}$ therefore $a^{-2}ab\in\mathcal{A}^{qnil}$, so $1+a^{-2}ab\in\mathcal{A}^{inv}.$ Hence $a+b\in\mathcal{A}^{inv}$ and $$(a+b)^{-1}=\left(a(1+a^{-2}ab)\right)^{-1}=(1+a^{-2}ab)^{-1}a^{-1}.$$
    \end{proof}
\end{theorem}
\begin{theorem}\label{inv_dra}
    Let $a\in\mathcal{A}^d$ and $b\in\mathcal{A}^{qnil}$, where $a,b$ are $\{a,b\}$-weakly commutative, then $a+b\in\mathcal{A}^d.$
    \begin{proof}
        Since $a,b$ are $\{a,b\}$-weakly commutative, then we have $ab=c_1a,ba=ac_1$ and $ab=bc_2,ba=c_2b$ for some $c_1,c_2\in\mathcal{A}$. Let $p=aa^d$ and $\mathcal{A}_1=p\mathcal{A}p,\mathcal{A}_2=(1-p)\mathcal{A}(1-p)$. Then corresponding to the idempotent $p$, $a$ has the matrix representation
        \begin{equation}\label{eqn1}
            a=\begin{bmatrix}
            a_1 & 0 \\
            0 & a_2 
        \end{bmatrix}_p ,\text{ where }a_1\in\mathcal{A}_1^{inv}\text{ and }a_2\in\mathcal{A}_2^{qnil}.
        \end{equation} Furthermore by Theorem \ref{alg_prod}, $b,c_1$ also have a  representation corresponding to (\ref{eqn1}), given by $$c_1=\begin{bmatrix}
            c_{11} & 0\\
            0 & c_{12}
        \end{bmatrix}_p \text{ and } b=\begin{bmatrix}
            b_1 & 0\\
            0 & b_2
        \end{bmatrix}_p,$$ where $b_1\in\mathcal{A}_1^{qnil},\text{ }b_2\in\mathcal{A}_2^{qnil}.$ Then $ab=c_1a,ba=ac_1$ gives $$a_1b_1=c_{11}a_1,\text{ }b_1a_1=a_1c_{11}\text{ and }a_2b_2=c_{12}a_2,\text{ }b_2a_2=c_{12}a_2.$$
       Moreover if $c_2=\begin{bmatrix}
           c_{21} & c_{22}\\
           c_{23} & c_{24}
           \end{bmatrix}_p$ then $ab=bc_2,ba=c_2b$ gives $$a_1b_1=b_1c_{21},\text{ }b_1a_1=c_{21}b_1\text{ and }a_2b_2=b_2c_{24},\text{ }b_2a_2=c_{24}b_2.$$ Therefore by Theorem \ref{inv_nil}, we get $a_1+b_1\in\mathcal{A}_1^{inv}$ and by Lemma \ref{bcommu}, $a_2+b_2\in\mathcal{A}_2^{qnil}.$ Hence $a+b\in\mathcal{A}^d.$
    \end{proof}
\end{theorem}
Using Theorem \ref{inv_nil} and Theorem \ref{inv_dra}, we obtain the following results similar to Theorem \ref{Drazin_main}.
\begin{theorem}
    Let $a,b\in\mathcal{A}^d$ such that $a,b$ are $\{a,b\}$-weakly commutative, then the following are equivalent:
    \begin{enumerate}[label=(\roman*)]
        \item $a+b\in \mathcal{A}^d;$ \label{con_11}
        \item $aa^d(a+b)bb^d\in \mathcal{A}^d;$\label{con_21}
        \item $aa^d(a+b)\in \mathcal{A}^d;$\label{con_31}
        \item $(a+b)bb^d \in \mathcal{A}^d.$\label{con_41}
    \end{enumerate}
\end{theorem}

\section*{Declarations}

The authors have no conflicts of interest to declare.

\makeatletter
\renewcommand\@makefnmark%
{\mbox{\textsuperscript{\normalfont\@thefnmark)}}}
\makeatother



\end{document}